\documentclass[a4paper,12pt]{article}

\usepackage{mathtext} 				
\usepackage[T2A]{fontenc}			
\usepackage[koi8-r]{inputenc}			
\usepackage{amsmath,amsfonts,amssymb,amsthm,mathtools, color} 
\theoremstyle{plain} 
\newtheorem{theorem}{Theorem}[section]
\newtheorem{lemma}{Lemma}[section]

\newtheorem{statement}[theorem]{Statement}
\newtheorem{proposition}[theorem]{Proposition}
\newtheorem{remark}{Remark}
 
\theoremstyle{plain}
\newtheorem*{theorem*}{Theorem}
\newtheorem*{lemma*}{Lemma}

\author{E.D. Galkovskii\footnote{Chebyshev Laboratory, St. Petersburg State University, 14th Line V.O., 29B, Saint Petersburg 199178 Russia. E-mail: egor\_maths@list.ru.}\,, 
A.I. Nazarov\footnote{St. Petersburg Department of V.A. Steklov Mathematical Institute, Russian Academy of Sciences; St. Petersburg State University. 
E-mail: al.il.nazarov@gmail.com. }}
\title{A general trace formula for the differential operator\\ on a segment with the last coefficient perturbed\\ by a finite signed measure }
\date{}

\begin{document}

\maketitle

\hfill {\it To the memory of M.Z. Solomyak}
\bigskip

\begin{abstract}
	A first order trace formula is obtained for a regular differential operator perturbed by a finite signed measure multiplication operator.
\end{abstract}

\section{Introduction}

Consider an operator $\mathbb L$ on a segment $[a,b]$ that is defined by a differential expression of order $n \geqslant 2$
\begin{align}\label{Diff_Equation_P}
\ell:=(-i)^nD^n+\sum \limits_{k=0}^{n-2}{p_k(x)D^k},
\end{align}
(here $p_k\in L_1(a,b)$ are complex-valued functions) and boundary conditions
\begin{align}\label{Boundary_Conditions}
	(P_j(D)y)(a)+(Q_j(D)y)(b) = 0,\qquad j=0, \dots, n-1.
\end{align}
Here $P_j$ and $Q_j$ are polynomials whose degrees do not exceed $n-1$. Let $d_j$ be the maximum of degrees of $P_j$ and $Q_j$. Suppose $a_j$ and $b_j$ are the $d_j$-th coefficients of $P_j$ and $Q_j$ respectively (therefore, $a_j$, $b_j$ cannot be zeros simultaneously). 

We assume that the system of boundary conditions (\ref{Boundary_Conditions}) is normalized , i.e. $\sum\limits_{j=0}^{n-1}d_j$ is minimal among all the systems of boundary condition that can be obtained from (\ref{Boundary_Conditions}) by linear bijective transformations. See \cite[ch. II, $\mathsection 4$]{Neimark} for a detailed explanation and \cite{Shk} for a more advance treatment.

We assume the boundary conditions (\ref{Boundary_Conditions}) to be Birkhoff regular, see \cite[ch. II, $\mathsection 4$]{Neimark}. Then the operator $\mathbb{L}$ has purely discrete spectrum\footnote{We underline that we do not require $\mathbb{L}$ to be self-adjoint.}, which we denote by $\{\lambda_N\}_{_{N=1}}^{^\infty}$. In what follows we always enumerate the eigenvalues in ascending order of their absolute values according to their multiplicities (that means $|\lambda_N| \leqslant |\lambda_{N+1}|$). 

Let $\mathfrak M[a,b]$ be the space of finite complex-valued measures. Denote by $\mathbb{Q}$ the operator of multiplication by ${\mathfrak q} \in \mathfrak M[a,b]$. Then the operator $\mathbb{L}_{\mathfrak q}=\mathbb{L}+\mathbb{Q}$ has also a purely discrete spectrum $\{\lambda_N({\mathfrak q})\}_{_{N=1}}^{^\infty}$.
    
    We are interested in the regularized trace
\begin{align*}
	\mathcal{S}({\mathfrak q}) := \sum_{N=1}^{\infty} \bigg[\lambda_N({\mathfrak q})-\lambda_N-\frac{1}{b-a}\int\limits_{[a,b]} {\mathfrak q}(dx)\,\bigg].
\end{align*}
Without loss of generality we suppose that $\int\limits_{[a,b]} {\mathfrak q}(dx)=0$.

    The first formula for a regularized trace was obtained by I.M. Gelfand and B.M. Levitan in 1953. In \cite{GL} they considered the problem    
\begin{equation}\label{SL-GL}
	-y''+{\mathfrak q}(x)y=\lambda y; \qquad y(0) = y(\pi) = 0
\end{equation}
and showed that for real-valued function ${\mathfrak q}(x)\in {\cal C}^1[0,\pi]$ the following relation holds:
$$
	\mathcal{S}({\mathfrak q})= -\frac{{\mathfrak q}(0)+{\mathfrak q}(\pi)}{4}.
$$

The paper \cite{GL} generated many improvements and generalizations, see a survey of V.A. Sadovnichii and V.E. Podolskii \cite{SPSurvey}.

In the recent work \cite{SZN} A.I. Nazarov, D.M. Stolyarov and P.B. Zatitskiy obtained formula 
\begin{equation}\label{traceNSZ}
	\mathcal{S}({\mathfrak q}) = \frac{\psi_a(a+)}{2n}\cdot\textbf{tr}\,(\mathbb A)+\frac{\psi_b(b-)}{2n}\cdot\textbf{tr}\,(\mathbb B),
\end{equation}
for arbitrary $n\geqslant 2$ and regular boundary conditions, under assumptions that are standard now\footnote{Formula (\ref{traceNSZ}) was earlier proved by R.F. Shevchenko \cite{Shv} for the operator $\mathbb{L}$ without lower-order terms and a smooth function ${\mathfrak q}$.}; namely, ${\mathfrak q} \in L_1(a,b)$ and the functions 
$$
\psi_a(x)=\frac{1}{x-a}\int\limits_a^x {\mathfrak q}(t)dt,\qquad 
\psi_b(x)=\frac{1}{b-x}\int\limits_x^b {\mathfrak q}(t)dt
$$ 
have bounded variations at points $a$ and $b$ respectively. In (\ref{traceNSZ})  $\mathbb A$ and $\mathbb B$ stand for the matrices with elements that can be expressed in terms of $a_j$ and $b_j$, $j=0,\dots,n-1$. Moreover, it was shown in \cite{SZN} that in important special case, where the boundary conditions are {\bf almost separated}, the values $\textbf{tr}\,(\mathbb A)$ and $\textbf{tr}\,(\mathbb B)$ in (\ref{traceNSZ}) can be reduced and expressed using only the sums of degrees of polynomials $P_j$ and $Q_j$ respectively.

Absolutely new phenomenon was discovered in our century by A.M. Savchuk and A.A. Shkalikov \cite{S2000, SSh}. Namely, let  ${\mathfrak q}\in \mathfrak M[0,\pi]$ be a
signed measure locally continuous at points $0$ and $\pi$. Then for the problem (\ref{SL-GL}) we have
\begin{equation}\label{SL-SSh}
	\mathcal{S}({\mathfrak q}) = -\frac{{\mathfrak q}(0)+{\mathfrak q}(\pi)}{4}-\frac{1}{8}\sum_j h_j^2,
\end{equation}
where $h_j$ stand for the jumps of the distribution function for the measure ${\mathfrak q}$. The series $\mathcal{S}({\mathfrak q})$ in this case is summed by mean-value method. 

Thus, for ${\mathfrak q}\in \mathfrak M[a,b]$ the regularized trace becomes non-linear functional of ${\mathfrak q}$. For $\delta$-potential this effect was slightly generalized in \cite{Konechnaya}. 

We generalize formula (\ref{SL-SSh}) for the operator $\mathbb{L}$ with arbitrary regular boundary conditions.\medskip

The paper is organized as follows. Section $\mathsection 2$ contains main results and some intermediate assertions. These assertions are proved in $\mathsection\mathsection 3-5$. The Appendix includes asymptotics of eigenvalues and eigenfunctions of Sturm-Liouville operators. These asymptotics are used in the proof of Theorem 2.4.

Let us introduce some notation. One can split a complex-valued measure ${\mathfrak q}\in \mathfrak M[a,b]$ into two parts -- continuous and discrete. We denote them by $\mathfrak c$ and $\mathfrak d$ respectively, so that
\begin{equation}\label{c+d}
\mathfrak q=\mathfrak c+\mathfrak d=\mathfrak c+\sum\limits_j h_j\delta(x-x_j), \qquad \sum\limits_j |h_j|<\infty.
\end{equation}
Denote by $\|\mathfrak q\|$ the total variation of $\mathfrak q$. We also define the distribution function
$$
{\cal Q}(x)=\int\limits_{[a,x]} {\mathfrak q}(dt).
$$ 
Thus, $h_j$ is the jump of ${\cal Q}$ at the point $x_j$.

Denote by $\mathbb{L}_0$ the operator generated by the differential expression $\ell_0=(-i)^nD^n$ and regular boundary conditions (\ref{Boundary_Conditions}). The eigenvalues of $\mathbb{L}_0$ are denoted by $\{\lambda_N^0\}_{_{N=1}}^{^{\infty}}$.
    
Further, $G(x,y,\lambda)$, $G_{\mathfrak q}(x,y,\lambda)$, and $G_0(x,y,\lambda)$ stand for the Green functions of operators $\mathbb{L}-\lambda$, $\mathbb{L}_{\mathfrak q}-\lambda$ and $\mathbb{L}_0-\lambda$ respectively, see \cite{Neimark}, ch. I, $\mathsection 3$.
    
For arbitrary function $\Phi(\lambda)$ defined on the complex plane $\mathbb{C}$, we introduce the function $\tilde{\Phi}(z)$ by the formula
\begin{equation*}
	\tilde{\Phi}(z) = \Phi(\lambda),\quad \text{ where }\quad z = \lambda^{\frac{1}{n}}, \; Arg(z) \in [0,\frac{2\pi}{n}).
\end{equation*}

Note that the resolvent $\frac 1{\mathbb{L}-\lambda}$ is an integral operator with a kernel $G(x,y,\lambda)$. So one can define the trace
$$
\mathbf{Sp}\,\frac{1}{{\mathbb L} - \lambda} = \int\limits_{a}^b G(x,x,\lambda)\, dx.
$$ 

   Recall the definition of summation by mean-value method (Ces\`aro summation of order~$1$). Let $I_{\ell}$ be a sequence of partial sums corresponding to the series $\sum\limits_{j}a_j$. The series is called mean-value summable if the following limit exists:
\begin{equation*}
({\cal C},1)\,\text{-}\lim_{\ell\to \infty} I_{\ell}:=({\cal C},1)\,\text{-}\sum \limits_{j=1}^{\infty} a_j := \lim_{k \to \infty} \frac{1}{k}\sum \limits_{\ell=1}^k I_{\ell}.
\end{equation*}

All positive constants whose exact values are not important are denoted by $C$.

\section{Formulation of results}

Our main result for the second order operators reads as follows:

\begin{theorem}\label{basic-1}
Suppose that $n=2$ and that the distribution function of the measure ${\mathfrak q}\in \mathfrak M[a,b]$ is differentiable at points $a$ and $b$. Let the boundary conditions (\ref{Boundary_Conditions}) be regular. Then the following formula holds:
\begin{equation}\label{SL-SSh-general}
	\mathcal{S}({\mathfrak q}) = {\cal A}{\cal Q}'(a)+{\cal B}{\cal Q}'(b)-\frac{1}{8}\sum_j h_j^2.
\end{equation}
Here the series $\mathcal{S}({\mathfrak q})$ is mean-value summable, and
\begin{eqnarray*}
{\cal A}={\cal B}=-\frac 14 & \text{if} & d_0=d_1=0;\\
{\cal A}={\cal B}=\frac 14 & \text{if} & d_0=d_1=1;\\
{\cal A}=-{\cal B}=\frac 14\,\frac {a_1b_0-a_0b_1}{a_1b_0+a_0b_1} & \text{if} & d_0=0,\ d_1=1.
\end{eqnarray*}
\end{theorem}
Thus, the nonlinear term in (\ref{SL-SSh-general}) {\bf does not depend} on boundary conditions while the coefficients of the linear term are completely determined by the boundary conditions.\medskip

For higher-order differential operators the perturbation considered is weak, and the dependency of the regularized trace on ${\mathfrak q}$ remains to be linear.

\begin{statement}
Suppose that $n\ge3$ and that ${\mathfrak q}\in \mathfrak M[a,b]$ is a measure subject to the conditions of Theorem \ref{basic-1}. Let the boundary conditions (\ref{Boundary_Conditions}) be regular. Then the following formula holds:
\begin{equation*}
	\mathcal{S}({\mathfrak q}) =\frac{{\cal Q}'(a)}{2n}\cdot\textbf{tr}\,(\mathbb A)+\frac{{\cal Q}'(b)}{2n}\cdot\textbf{tr}\,(\mathbb B).
\end{equation*}
Here the series $\mathcal{S}({\mathfrak q})$ are mean-value summable, and the matrices $\mathbb A$ and $\mathbb B$ are the same as in (\ref{traceNSZ}), see \cite[Theorem 2]{SZN}.
\end{statement}

This statement will be proved in full generality in a forthcoming paper. Here we prove Theorem  \ref{basic-1} and some auxiliary statements.

\begin{theorem}\label{th1}
	For every sequence $R = R_{\ell} \to \infty$ separated from $|\lambda_N^0|^{\frac{1}{n}}$ the following relations hold:
    \begin{enumerate}
\item if $n \geqslant 3$ then
	\begin{equation*}
\sum_{\lambda_N({\mathfrak q}),\lambda_N<R} \Big[\lambda_N({\mathfrak q})-\lambda_N\Big]
 = -\frac{1}{2 \pi i}\int\limits_{|\lambda|=R^n}\int\limits_{[a,b]} G_0(x,x,\lambda)\,{\mathfrak q}(dx)\,d\lambda + o(1);
	\end{equation*}
\item if $n = 2$ then
    \begin{eqnarray}\label{Splitting}
\sum_{\lambda_N({\mathfrak q}),\lambda_N<R} \Big[\lambda_N({\mathfrak q})-\lambda_N\Big]&=&
-\frac{1}{2 \pi i}\int\limits_{|\lambda|=R^2}\int\limits_{[a,b]} G_0(x,x,\lambda)\,{\mathfrak q}(dx)\,d\lambda 
\nonumber \\ 
&+& \frac{1}{4 \pi i}\sum_j h_j^2\int\limits_{|\lambda|=R^2}\! G_0(x_j,x_j,\lambda)^2\,d\lambda + o(1).
	\end{eqnarray}
\end{enumerate}
\end{theorem}

\begin{theorem}\label{th2}
Suppose that $n=2$ and that ${\mathfrak q} \in \mathfrak M[a,b]$ is a measure subject to the conditions of Theorem \ref{basic-1}. Let $R = R_{\ell} \to \infty$ be a sequence separated from $|\lambda_N^0|^{\frac{1}{2}}$ such that:
\begin{enumerate}
\item if the boundary conditions (\ref{Boundary_Conditions}) are strongly regular (see, e.g., \cite{Shk}) then for $\ell$ large enough there is exactly one value $|\lambda_N^0|^{\frac{1}{2}}$ between $R_{\ell}$ and $R_{\ell+1}$;
\item if the boundary conditions (\ref{Boundary_Conditions}) are regular, but not strongly regular, then for $\ell$ large enough there is exactly one pair of values $|\lambda_N^0|^{\frac{1}{2}}$ between $R_{\ell}$ and $R_{\ell+1}$.
\end{enumerate}
 
Then
\begin{equation}\label{C-lim}
-\frac{1}{2 \pi i}\cdot({\cal C},1)\text{-}\lim \int\limits_{|\lambda|=R^n}\int\limits_{[a,b]}G_0(x,x,\lambda)\,{\mathfrak q}(dx)\,d\lambda = {\cal A}{\cal Q}'(a)+{\cal B}{\cal Q}'(b),
\end{equation}
where ${\cal A}$ and ${\cal B}$ are the same as in Theorem \ref{basic-1}.
\end{theorem}

\begin{theorem}\label{th3}
  Suppose that $n=2$ and that $x\ne a,b$. Then for every sequence $R = R_{\ell} \to \infty$ separated from $|\lambda_N^0|^{\frac{1}{2}}$ we have
\begin{equation*}
	\lim_{R\to \infty}\int\limits_{|\lambda|=R^2}G_0(x,x,\lambda)^2\,d\lambda = -\frac{\pi i}{2}.
\end{equation*}
\end{theorem}

\section{Proof of Theorem \ref{th1}}

We use some statements obtained in \cite{SZN}. The first statement generalizes the Tamarkin equiconvergence Theorem \cite{T}, the second one provides estimates of the Green functions.

\begin{proposition}\label{Equiconv_G_G0} (\cite[Theorem 1]{SZN}) For every sequence $R = R_{\ell} \to \infty$ separated from $|\lambda_N^0|^{\frac{1}{n}}$ the relation 
\begin{equation*}
	\int \limits _{|\lambda| = R^n}{\left|(G_0 - G)(x,y,\lambda)\right||d\lambda|}\to 0
\end{equation*}
holds uniformly in  $x, y \in [a,b]$.
\end{proposition}

\begin{proposition}\label{Estimates_G_0} (\cite[Lemma 1 and (22)]{SZN})
Put 
$$
\Gamma_1=\Big\{w=e^{i\phi}:\phi\in\Big(0,\,\frac{\pi}{n}\Big)\Big\}; \qquad 
\Gamma_2=\Big\{w=e^{i\phi}:\phi\in\Big(\frac{\pi}{n},\,\frac{2\pi}{n}\Big)\Big\}.
$$
Then for every $x \in [a,b]$ 
\begin{equation}\label{G0to0}
R^{n-1}\cdot|\tilde{G_0}(x,y,Rw)| \to 0
\end{equation}
for almost all $y\in[a,b]$ and almost all $w\in\Gamma_1\cup\Gamma_2$. Moreover, the convergence is uniform on the set ${\cal K}\times {\cal J}$ for every compact set ${\cal K} \subset [a,b]^2$, separated from corners and diagonal $\{x=y\}$ and every compact set ${\cal J}\subset\Gamma_1\cup\Gamma_2$.

Further, assume that all coefficients $p_k$, $k=0,\dots,n-2$, in the differential expression (\ref{Diff_Equation_P}) belong to the space $\mathfrak M[a,b]$. Then for every sequence $R = R_{\ell} \to \infty$ separated from $|\lambda_N^0|^{\frac{1}{n}}$ and for all $j=0,\dots,n-1$ the functions 
$$
R^{n-1-j}\cdot|(\tilde{G})_x^{(j)}(x,y,Rw)|
$$
are uniformly bounded in $[a,b]^2\times(\Gamma_1\cup\Gamma_2)$. 
\end{proposition}
\begin{remark}
The second part of this statement is proved in \cite{SZN} for $p_k\in L_1(a,b)$. However, the proof runs without changes for $p_k\in\mathfrak M[a,b]$. 
\end{remark}

\begin{proof}[Proof of Theorem \ref{th1}]

We will start from the relation (see \cite[(24), (25)]{SZN}):
\begin{eqnarray*}
&&4 \pi i\sum_{\lambda_N({\mathfrak q}),\lambda_N<R} \Big[\lambda_N({\mathfrak q})-\lambda_N\Big] \\
&& = -\int\limits_{|\lambda|=R^n}\lambda\,\mathbf{Sp}\left(\left(\frac{1}{\mathbb{L}-\lambda}\,\mathbb{Q}\,\frac{1}{\mathbb{L}_{\mathfrak q}-\lambda}\right)\mathbb{Q}\left(\frac{1}{\mathbb{L}-\lambda}\,\mathbb{Q}\,\frac{1}{\mathbb{L}_{\mathfrak q} -\lambda}\right)\right)d\lambda \\
&& + \int\limits_{|\lambda|=R^n}\lambda\,\mathbf{Sp}\left(\frac{1}{\mathbb{L}-\lambda}\,\mathbb{Q}\,\frac{1}{\mathbb{L}-\lambda}+\frac{1}{\mathbb{L}_{\mathfrak q}-\lambda}\,\mathbb{Q}\,\frac{1}{\mathbb{L}_{\mathfrak q}-\lambda}\right)d\lambda =: -I_1(R)+I_2(R).
\end{eqnarray*}
      
\begin{lemma}\label{I1}
Under assumptions of Theorem \ref{th1} we have $I_1(R)=o(1)$ as $R\to\infty$.
\end{lemma}
\begin{proof}
We rewrite $I_1(R)$ in terms of the Green functions:
\begin{multline*}
	I_1(R)= \int\limits_{|\lambda|=R^n} \int\limits_{\vphantom{[}a}^b \int\limits_{[a,b]}\lambda\, G(x,y,\lambda)\\
\times\int\limits_{[a,b]}\int\limits_{[a,b]} G_{\mathfrak q}(y,s,\lambda)\,G(s,t,\lambda)\,G_{\mathfrak q}(t,x,\lambda){\mathfrak q}(dt){\mathfrak q}(ds){\mathfrak q}(dy)\,dx\,d\lambda.
\end{multline*}

Let $n \geqslant 3$. Then the estimate from Proposition \ref{Estimates_G_0} implies
$$
|I_1(R)|\leqslant R^{2n}\|{\mathfrak q}\|^3\,\frac{C}{R^{4(n-1)}} = o(1).
$$
For $n = 2$ the proof is more complicated. From the same estimate we obtain for $x,y\in[a,b]$
\begin{equation}\label{ThreeFour}
\bigg|\lambda\int\limits_{[a,b]}\int\limits_{[a,b]}G_{\mathfrak q}(y,s,\lambda)\,G(s,t,\lambda)\,G_{\mathfrak q}(t,x,\lambda){\mathfrak q}(dt){\mathfrak q}(ds) \bigg| \leqslant \|{\mathfrak q}\|^2\,\frac{C}{R},
\end{equation}
whence
\begin{eqnarray*}
|I_1(R)|&\leqslant& \|{\mathfrak q}\|^2\,\frac{C}{R}\int\limits_{|\lambda|=R^2}\int\limits_{\vphantom{[}a}^b \int\limits_{[a,b]}|G(x,y,\lambda)|\,|{\mathfrak q}|(dy)\,dx\,|d\lambda| \\
&\stackrel{*}{=}&\|{\mathfrak q}\|^2\,\frac{C}{R}\int\limits_{|\lambda|=R^2}\int\limits_{\vphantom{[}a}^b \int\limits_{[a,b]}|G_0(x,y,\lambda)|\,|{\mathfrak q}|(dy)\,dx\,|d\lambda| +o(1) \\
&\leqslant& \|{\mathfrak q}\|^2\, C\int\limits_{\Gamma_1\cup\Gamma_2}\int\limits_{\vphantom{[}a}^b\int\limits_{[a,b]}R\cdot|\tilde{G}_0(x,y,Rw)|\,|{\mathfrak c}|(dy)\,dx\,|dw| \\
&+&\sum_j\|{\mathfrak q}\|^2\,C\,|h_j| \int\limits_{\Gamma_1\cup\Gamma_2}\int \limits_a^b 
R\cdot|\tilde{G}_0(x,x_j,Rw)|\,dx\,|dw|+o(1) \\
&=:&I_{11}(R)+I_{12}(R)+o(1)
\end{eqnarray*}
(the relation $*$ follows from Proposition \ref{Equiconv_G_G0}). 

Due to the estimates from Proposition \ref{Estimates_G_0} the integrand in $I_{12}(R)$ is bounded uniformly in $j$. Moreover, under assumptions of Theorem 2.3 the measure ${\mathfrak q}$ has no atoms at the endpoints of the segment, i.e. $x_j\in(a,b)$. Hence the relation (\ref{G0to0}) and the Lebesgue Dominated Convergence Theorem imply $I_{12}(R)=o(1)$. 

Now we estimate $I_{11}(R)$. Since ${\mathfrak c}$ is continuous, for every $\varepsilon > 0$ there exists $\delta > 0$ such that for every segment with the length less than $\delta$, the total variation of ${\mathfrak c}$ on this segment does not exceed $\varepsilon$. We choose a compact set ${\cal K} \subset [a,b]^2$ separated from corners and diagonal $\{x=y\}$ so that for every $x\in[a,b]$ the set ${\cal K}_x = \{y\in[a,b]:(x,y) \notin {\cal K}\}$ is a conjunction of three or less intervals with the length less than $\delta$. Also we choose a compact set ${\cal J} \subset \Gamma_1\cup\Gamma_2$ so that the measure of $(\Gamma_1\cup\Gamma_2) \setminus {\cal J}$ does not exceed $\varepsilon$.

The integral over ${\cal K}\times {\cal J}$ tends to zero as $R\to \infty$ by Proposition \ref{Estimates_G_0}. The integral over the remainder set can be estimated by $C\varepsilon$. 

Thus, $|I_1(R)|\leqslant C\varepsilon + o(1)$. Since $\varepsilon$ is arbitrarily small, the statement follows.
\end{proof}

We continue the proof of Theorem \ref{th1}. Using the relation $\mathbf{Sp}(ABC)=\mathbf{Sp}(BCA)$ and integrating by parts, we rewrite $I_2(R)$ as follows:
\begin{eqnarray*}
I_2(R) &=& \int\limits_{|\lambda|=R^n}\lambda\,\mathbf{Sp}\left(\frac{1}{\mathbb{L}-\lambda}\,\mathbb{Q}\,\frac{1}{\mathbb{L}-\lambda}+\frac{1}{\mathbb{L}_{\mathfrak q} -\lambda}\,\mathbb{Q}\,\frac{1}{\mathbb{L}_{\mathfrak q}-\lambda}\right)d\lambda \\
&=& \int\limits_{|\lambda|=R^n}\mathbf{Sp}\left(\left(\frac{\lambda}{(\mathbb{L}-\lambda)^2}+\frac{\lambda}{(\mathbb{L}_{\mathfrak q}-\lambda)^2}\right)\mathbb{Q}\right)d\lambda \\
&=& -\int\limits_{|\lambda|=R^n}\mathbf{Sp}\left(\left(\frac{1}{\mathbb{L}-\lambda}+\frac{1}{\mathbb{L}_{\mathfrak q}-\lambda}\right)\mathbb{Q}\right)d\lambda.
\end{eqnarray*}

We apply the Hilbert resolvent identity to the second term and obtain
\begin{equation}\label{I2}
I_2(R) = -2\int\limits_{|\lambda|=R^n}\mathbf{Sp}\left(\frac{1}{\mathbb{L}-\lambda}\,\mathbb{Q}\right)d\lambda+\int\limits_{|\lambda|=R^n}\mathbf{Sp}\left(\frac{1}{\mathbb{L}-\lambda}\,\mathbb{Q}\,\frac{1}{\mathbb{L}_{\mathfrak q}-\lambda}\,\mathbb{Q}\right)d\lambda.
\end{equation}
	By the estimate from Proposition \ref{Estimates_G_0}, for $n \geqslant 3$ the second term in (\ref{I2}) is bounded by $CR^{2-n}=o(1)$. We rewrite the first term in terms of the Green function and obtain
\begin{equation*}
	I_2(R) = -2 \int\limits_{|\lambda|=R^n} \int\limits_{[a,b]} G(x,x,\lambda){\mathfrak q}(dx)\,d\lambda+o(1)= -2 \int\limits_{|\lambda|=R^n} \int\limits_{[a,b]} G_0(x,x,\lambda){\mathfrak q}(dx)\,d\lambda+o(1)
\end{equation*}
(the last equality follows from Proposition \ref{Equiconv_G_G0}). This equality and Lemma \ref{I1} give the first statement of Theorem.\medskip

If $n=2$ we apply the Hilbert resolvent identity to the second term in (\ref{I2}) and obtain
\begin{equation*}
I_2(R) = \int\limits_{|\lambda|=R^2}\mathbf{Sp}\left(\frac{-2}{\mathbb{L}-\lambda}\,\mathbb{Q}+\left(\frac{1}{\mathbb{L}-\lambda}\,\mathbb{Q}\right)^2-\left(\frac{1}{\mathbb{L}-\lambda}\,\mathbb{Q}\right)^2\frac{1}{\mathbb{L}_{\mathfrak q}-\lambda}\,\mathbb{Q}\right)d\lambda.
\end{equation*}
By the estimate from Proposition \ref{Estimates_G_0}, the last term here is bounded by $CR^{-1}=o(1)$. We rewrite the remaining terms via the Green function and replace $G$ by $G_0$ similarly to case $n \geqslant 3$. Thus we arrive at
\begin{eqnarray}\label{I3}
	I_2(R) &=& -2 \int\limits_{|\lambda|=R^2} \int\limits_{[a,b]} G_0(x,x,\lambda){\mathfrak q}(dx)\,d\lambda
    \nonumber\\
&+& \int\limits_{|\lambda| = R^2} \int\limits_{[a,b]} \int\limits_{[a,b]} G_0(x,y,\lambda)\,G_0(y,x,\lambda){\mathfrak q}(dy){\mathfrak q}(dx)\,d\lambda
+o(1).
\end{eqnarray}

It remains to simplify the second term in (\ref{I3}). We denote it by $I_3(R)$ and rewrite as follows:
\begin{eqnarray*}\label{I31}
	I_3(R) &=& \int\limits_{|\lambda| = R^2} \int\limits_{[a,b]} \int\limits_{[a,b]} G_0(x,y,\lambda)\,G_0(y,x,\lambda){\mathfrak c}(dy)({\mathfrak c}(dx) +2{\mathfrak d}(dx))\,d\lambda     \nonumber\\
&+& \int\limits_{|\lambda| = R^2} \int\limits_{[a,b]} \int\limits_{[a,b]} G_0(x,y,\lambda)\,G_0(y,x,\lambda){\mathfrak d}(dy){\mathfrak d}(dx)\,d\lambda =:
I_{31}(R)+I_{32}(R). 
\end{eqnarray*}

The integral $I_{31}(R)$ can be estimated in the same way as $I_{11}(R)$. This gives $|I_{31}(R)|\leqslant C\varepsilon + o(1)$ for any $\varepsilon>0$.

Further, we have
\begin{equation*}
I_{32}(R) = \sum_{j,k} h_kh_j
\int\limits_{\Gamma_1\cup\Gamma_2}\! 2R^2\cdot \tilde{G}_0(x_j,x_k,Rw)
\,\tilde{G}_0(x_k,x_j,Rw)\,dw.
\end{equation*}
By (\ref{G0to0}), all terms with $j\ne k$ tend to zero as $R\to\infty$. Using the Lebesgue Theorem we obtain 
\begin{equation*}
I_{32}(R) = \sum_{j} h_j^2
\int\limits_{\Gamma_1\cup\Gamma_2}\! 2R^2\cdot \tilde{G}_0(x_j,x_j,Rw)^2\,dw + o(1) =
\sum_{j} h_j^2 \int\limits_{|\lambda| = R^2} G_0(x_j,x_j,\lambda)^2\,d\lambda+ o(1).
\end{equation*}
This relation, formula (\ref{I3}) and estimates of $I_1$ and $I_{31}$ give us (\ref{Splitting}).
\end{proof}

 \section{Proof of Theorem \ref{th2}}

Changing variables we can assume $a=0$, $b=1$. Since formula (\ref{C-lim}) is known for smooth $\mathfrak q$,  it is sufficient to prove Theorem in the case ${\cal Q}'(0)={\cal Q}'(1)=0$. Moreover, we can impose some additional orthogonality conditions on ${\mathfrak q}$. This notice will be used later.

We also need the following statement, which is a particular case of \cite[Лемма 1]{SSh}.

\begin{proposition}\label{cosinusy}
Suppose that the measure ${\mathfrak q} \in \mathfrak M[0,1]$ satisfies assumptions of Theorem \ref{basic-1} and that ${\cal Q}'(0)={\cal Q}'(1)=0$. Then
$$
({\cal C},1)\text{-}\sum \limits_{\ell=1}^{\infty}\, \int\limits_{[0,1]} \cos(2\pi\ell x)\,{\mathfrak q}(dx)=0. 
$$
\end{proposition}

\begin{proof}[Proof of Theorem \ref{th2}]
Expanding the Green function in a neighborhood of a pole, see \cite[ch. I, $\mathsection 3$]{Neimark}, and using the residue theorem we rewrite the integral in (\ref{C-lim}) as follows:
\begin{equation}\label{summa}
-\frac{1}{2 \pi i}\int\limits_{|\lambda|=R^n}\int\limits_{[0,1]}G_0(x,x,\lambda)\,{\mathfrak q}(dx)\,d\lambda=\sum\limits_{|\lambda_N^0| < R^2}\, \int\limits_{[0,1]} y_N(x) \overline{z_N(x)}\,{\mathfrak q}(dx).
\end{equation}
Here $y_N$ and $z_N$ denote the eigenfunctions of operators $\mathbb{L}_0$ and $\mathbb{L}_0^*$ corresponding to the eigenvalues $\lambda_N^0$ and $\overline{\lambda_N^0}$ respectively and normalized as follows:
$$
\langle y_N, z_N \rangle:=\int\limits_0^1 y_N(x)\overline{z_N(x)}\,dx = 1.
$$
If the eigenvalue $\lambda_N^0=\lambda_{N+1}^0$ corresponds to a two-dimensional Jordan block (in this case the same is true for $\overline{\lambda_N^0}=\overline{\lambda_{N+1}^0}$), then the term $y_N(x) \overline{z_N(x)}$ in the right-hand side of (\ref{summa}) should be replaced by
$$
y_N(x) \overline{\widehat z_{N+1}(x)}+
\widehat y_{N+1}(x) \overline{z_N(x)}.
$$
Here $\widehat y_{N+1}$ and $\widehat z_{N+1}$ stand for the adjoined functions of $\mathbb{L}_0$ and $\mathbb{L}_0^*$ in these Jordan blocks, and the normalization condition has the form
\begin{equation}\label{prisoed}
\langle y_N, z_N \rangle = \langle \widehat y_{N+1}, \widehat z_{N+1} \rangle = 0; \qquad
\langle y_N, \widehat z_{N+1} \rangle = \langle \widehat y_{N+1}, z_N \rangle = 1.
\end{equation}
Thus we need to justify the passage to the limit in the sense of mean-value in the right-hand side of (\ref{summa}). We consider several cases.

\subsection*{The case $d_0=d_1=0$ (the Dirichlet boundary conditions)\footnote{As we mentioned in the Introduction this case was considered in the paper \cite{SSh}.}}
This case is the simplest technically. The system of boundary conditions (\ref{Boundary_Conditions}) can be reduced to the form
$$  
y(0) = 0,\qquad y(1) = 0.
$$
The operator $-D^2$ with these boundary conditions is selfadjoint, its eigenvalues and eigen\-functions are as follows: 
$$
\lambda_N^0 = (\pi N)^2,\qquad   y_N(x) = z_N(x) = C\sin(\pi N x),
\qquad N\in\mathbb{N}.
$$
Taking into account the normalization condition we have
$$
y_N(x)\overline{z_N(x)} = 1-\cos(2 \pi N x).
$$ 
The constant vanishes after integration in view of the assumption $\int\limits_{[a,b]} {\mathfrak q}(dx)=0$ while the cosine disappears after passage to the limit due to Proposition \ref{cosinusy}. Thus formula (\ref{C-lim}) is proved.

\subsection*{The case $d_0=d_1=1$}

In this case the system (\ref{Boundary_Conditions}) can be rewritten as follows:
\begin{equation*}
\left\{
\begin{array}{lr}
  y'(0) + c_0 y(0) + f_0 y(1) = 0,\\
  y'(1) + c_1 y(0) + f_1 y(1) = 0,
\end{array}
\right.
\end{equation*}
and the boundary conditions of the adjoint operator have the form
\begin{equation*}
\left\{
\begin{array}{lr}
  z'(0) + \overline{c}_0 z(0) - \overline{c}_1 z(1) = 0,\\
  z'(1) - \overline{f}_0 z(0) + \overline{f}_1 z(1) = 0.
\end{array}
\right.
\end{equation*}

Using the algorithm in \cite[Ch. II, $\mathsection 4$]{Neimark} we write down the asymptotic expansions for eigenvalues and eigenfunctions up to $O(N^{-2})$, 
see Appendix, part 1. Taking into account the normalization condition we obtain
\begin{eqnarray}\label{Neumann}
	y_{N+1}(x)\overline{z_{N+1}(x)} &=& 1+\cos(2 \pi N x) -2\,\frac{\sin(2 \pi N x)}{\pi N}\,(c_0 (1-x)+ f_1 x) 
    \nonumber\\ 
    &+& (-1)^N\,\frac{\sin(2 \pi N x)}{\pi N}\,(c_1-f_0)(1-2 x) + O(N^{-2}).
\end{eqnarray}
Similarly to the Dirichlet case, the first two terms of this expansion disappear after integration and passage to the limit in (\ref{summa}). The other terms in (\ref{Neumann}) generate the series converging at every point of the segment $[0,1]$. Moreover, the partial sums of this series are uniformly bounded. Denote the sum of this series by $g(x)$. Then the Lebesgue Theorem gives
$$
({\cal C},1)\,\text{-}\sum\limits_{N=1}^{\infty}\, \int\limits_{[0,1]} y_N(x) \overline{z_N(x)}\,{\mathfrak q}(dx)=
\int\limits_{[0,1]} g(x)\,{\mathfrak q}(dx).
$$
  But we already know that for smooth ${\mathfrak q}$ the left-hand side of this equality equals zero. This implies $g(x)=0$ for a.e. $x\in[0,1]$.

It remains to notice that according to well-known formula \cite[5.4.2.9]{Prudnikov}, after summation the third term in (\ref{Neumann}) gives a function continuous except maybe for the endpoints of the segment. Next, by
\cite[5.4.2.10]{Prudnikov} the fourth term gives after summation a continuous function (the discontinuity at the point $\frac 12$ disappears because of the factor $1-2x$). The remainder term also gives a continuous function. Therefore $g$ can differ from zero only at points $0$ and $1$. However, by the assumptions imposed on ${\mathfrak q}$ these points do not contribute to the integral. This completes the proof of formula (\ref{C-lim}).

\subsection*{The case $d_0=0$, $d_1=1$}

A general form of the boundary conditions in this case is as follows:
\begin{equation*}
\left\{
\begin{array}{r}
a_0y(0) + b_0y(1) = 0,\\
a_1y'(0) + b_1y'(1)+ c_1y(0) + f_1y(1) = 0.
\end{array}
\right.
\end{equation*}
Without loss of generality we can assume $a_1 \neq 0$. Then the boundary conditions of the adjoint operator have the form
\begin{equation*}
\left\{
\begin{array}{r}
\displaystyle\overline{b}_0y'(0) + \overline{a}_0y'(1)+ \frac 1{\overline{a}_1}\,(\overline{c_1b}_0-\overline{f_1a}_0)y(0)= 0,\\
\overline{b}_1y(0) + \overline{a}_1y(1) = 0.
\end{array}
\right.
\end{equation*}

We introduce the following  notation:
\begin{equation*}
\mathfrak A = b_1 a_0 + a_1 b_0;\qquad \mathfrak B = f_1 a_0 - c_1 b_0;\qquad \mathfrak C = a_1 a_0 + b_1 b_0
\end{equation*}
(recall that $\mathfrak A\ne0$ by regularity of boundary conditions).

\subsubsection*{The case $\mathfrak C=\pm \mathfrak A$, $\mathfrak B = 0$: double eigenvalues}

In this case the system of boundary conditions for the operator $\mathbb{L}_0$ can be simplified as follows:
\begin{equation*}
a_0y(0) + b_0y(1) = 0,\qquad a_1y'(0) + b_1y'(1) = 0.
\end{equation*}
Suppose that\footnote{In the case $\mathfrak C = \mathfrak A$ all formulas are quite similar if we write the asymptotic expansions in powers of $N-\frac 12$.} $\mathfrak C = -\mathfrak A$. 
Then we have three variants: 
\begin{itemize}
	\item[1.] $a_1+b_1=0$, $a_0+b_0=0$;
    \item[2.] $a_1+b_1=0$, $a_0+b_0\neq 0$;
    \item[3.] $a_1+b_1\neq 0$, $a_0+b_0=0$.
\end{itemize}
It is easy to see that the third variant can be obtained from the second one by substitution $\mathbb{L}_0^*$ for $\mathbb{L}_0$. 

\paragraph{Variant $a_1+b_1=0$, $a_0+b_0=0$: no Jordan blocks.}
In this simple case the boundary conditions are reduced to the periodic ones:
\begin{equation*}
y'(0) - y'(1) = 0, \qquad y(0) - y(1) = 0.
\end{equation*}
The operator $\mathbb{L}_0$ with these boundary conditions is selfadjoint. Its eigenvalues and eigen\-functions have the form
\begin{eqnarray*}
\lambda_1^0 = 0, &&   y_1(x)=z_1(x)\equiv 1;\\
\lambda_{2N}^0=\lambda_{2N+1}^0=(2\pi N)^2, &&
y_{2N}(x) = z_{2N}(x) = C\sin(2\pi N x),\\
&& y_{2N+1}(x) = z_{2N+1}(x) = C\cos(2\pi N x),\qquad N\in\mathbb{N}.
\end{eqnarray*}
Taking into account the normalization condition we have
$$
y_{2N}(x)\overline{z_{2N}(x)} = 1-\cos(4 \pi N x);\qquad
y_{2N+1}(x)\overline{z_{2N+1}(x)} = 1+\cos(4 \pi N x).
$$ 
The pairwise summation yields a constant that disappears after the integration. Formula (\ref{C-lim}) is obvious.

\paragraph{Variant $a_1+b_1=0$, $a_0+b_0\neq0$: Jordan blocks.}
The boundary conditions have the form
\begin{eqnarray*}
\mathbb{L}_0: & a_0y(0) + b_0y(1) = 0, & y'(0) -y'(1) = 0;\\
\mathbb{L}_0^*: & \overline{b}_0z'(0) + \overline{a}_0z'(1) = 0, & z(0) - z(1) = 0.
\end{eqnarray*}

We write down the eigenvalues and eigenfunctions, see Appendix, part 2. Taking into account the condition (\ref{prisoed}) we obtain
\begin{equation}\label{cos4Nx}
y_{2N}(x)\overline{\widehat z_{2N+1}(x)}+\widehat y_{2N+1}(x) \overline{z_{2N}(x)} = 2+2\cos(4 \pi N x)\, \frac{(a_0 + b_0)(2x-1)} {a_0-b_0}. 
\end{equation}
Subtracting a proper smooth function we can assume that ${\mathfrak q}$ satisfies additional conditions  
\begin{equation}\label{2x-1}
\int\limits_{[0,\frac 12]} (2x-1)\,{\mathfrak q}(dx)=\int\limits_{[\frac 12,1]} (2x-1)\,{\mathfrak q}(dx)=0.
\end{equation}
Then the measure $\widetilde{\mathfrak q}(dx)=(2x-1)\,{\mathfrak q}(dx)$ satisfies the assumptions of Proposition \ref{cosinusy} on segments $[0,\frac 12]$ and $[\frac 12,1]$. Therefore the right-hand side in (\ref{cos4Nx}) vanishes after integration and passage to the limit in (\ref{summa}). Since (\ref{2x-1}) implies $\int\limits_{[0,1]} y_1(x)\overline{z_1(x)}\,{\mathfrak q}(dx)=0$, formula (\ref{C-lim}) is proved.

\subsubsection*{The case $\mathfrak C=\pm \mathfrak A$, $\mathfrak B \neq 0$: 
asymptotically close eigenvalues}

As in the previous case we suppose that $\mathfrak C = -\mathfrak A$ (the case $\mathfrak C = \mathfrak A$ is similar). Then the assumptions on $\mathfrak A$, $\mathfrak B$, $\mathfrak C$ can be rewritten as follows: 
$$
 (a_1 + b_1)(a_0 + b_0) = 0;\qquad f_1 a_0 - c_1 b_0 \neq 0.
$$
We again have three variants:  
\begin{itemize}
	\item $a_0 + b_0 = 0$, $a_1 + b_1 \neq 0$, $c_1 + f_1 \neq 0$;
	\item $a_0 + b_0 \neq 0$, $a_1 + b_1 = 0$;
	\item $a_0 + b_0 = 0$, $a_1 + b_1 = 0$, $c_1 + f_1 \neq 0$.
\end{itemize}
One can easily see that the second variant can be obtained from the first one by substitution $\mathbb{L}_0$ for $\mathbb{L}_0^*$.

\paragraph{The first variant.}
We write down the asymptotic expansions for eigenvalues and eigen\-functions up to $O(N^{-4})$, 
see Appendix, part 3.1, and take into account the normalization condition. Combining pairwise the terms corresponding to asymptotically close eigenvalues we obtain
\begin{eqnarray*}
y_{2N}(x)\overline{z_{2N}(x)} &+& 
y_{2N+1}(x)\overline{z_{2N+1}(x)} =
\eta_N^{+}(x)\overline{\zeta_N^{+}(x)} +\eta_N^{-}(x)\overline{\zeta_N^{-}(x)}
\\
= 2 &+&  2\cos(4\pi N x)\,\frac{(a_1 + b_1)(1-2x)}{a_1-b_1} \\
&+& 2 \sin(4\pi N x)\,\frac{(c_1+f_1)(1-2x)(b_1x-a_1(1-x))}{(a_1-b_1)^2 \pi N} + O(N^{-2}).
\end{eqnarray*}
The first two terms are summed up as in formula (\ref{cos4Nx}) and the last two ones -- as in (\ref{Neumann}). Formula (\ref{C-lim}) is proved.

\paragraph{The third variant.}
In this variant the system of boundary conditions for the operator $\mathbb{L}_0$ can be reduced:
\begin{equation*}
y(0) - y(1) = 0,\qquad y'(0) - y'(1) +c_1y(0) = 0.
\end{equation*}
We write down the asymptotic expansions for eigenvalues and eigenfunctions up to $O(N^{-6})$, 
see Appendix, part 3.2, and take into account the normalization condition. Combining pairwise the terms corresponding to asymptotically close eigenvalues we obtain
\begin{eqnarray*}
y_{2N}(x)\overline{z_{2N}(x)} + 
y_{2N+1}(x)\overline{z_{2N+1}(x)} &=&
\eta_N^{+}(x)\overline{\zeta_N^{+}(x)} +\eta_N^{-}(x)\overline{\zeta_N^{-}(x)}\\
&=& 2 + \sin(4\pi N x)\, \frac{c_1(2x-1)}{2\pi N} + O(N^{-2}).
\end{eqnarray*}
This series can be summed up as in formula (\ref{Neumann}). Formula (\ref{C-lim}) is proved.

\subsubsection*{Strongly regular case $\mathfrak C\neq \pm \mathfrak A$}

To simplify the proof we use the following obvious lemma:

\begin{lemma}\label{pairwise}
Let $\lim\limits_{k\to\infty}\frac{a_k}{k}=0$. Then
\begin{eqnarray}
({\cal C},1)\text{-}\sum\limits_{k=1}^{\infty} a_k &=& ({\cal C},1)\text{-}\sum\limits_{k=1}^\infty (a_{2k-1}+a_{2k})-\frac 12\cdot({\cal C},1)\text{-}\lim\limits_{k\to \infty} a_{2k}
\nonumber\\
&=& ({\cal C},1)\text{-}\sum\limits_{k=1}^\infty (a_{2k}+a_{2k+1})+a_1-\frac 12\cdot({\cal C},1)\text{-}\lim_{k\to \infty} a_{2k+1},
\label{odd-even}
\end{eqnarray}
i.e. if one of the expressions in the right-hand side of (\ref{odd-even}) converges then the series in the left-hand side converges.
\end{lemma}

We write down the asymptotic expansions for eigenvalues and eigenfunctions up to $O(N^{-2})$, 
see Appendix, part 4. It is easy to see that $\lim\limits_{N\to \infty}\frac 1N\,y_N(x)\overline{z_N(x)}=0$, so we can apply the second part of formula (\ref{odd-even}). We start from pairwise sums
\begin{eqnarray*}
y_{2N}(x)\overline{z_{2N}(x)} &+& 
y_{2N+1}(x)\overline{z_{2N+1}(x)}
=
\eta_N^{+}(x)\overline{\zeta_N^{+}(x)} +\eta_N^{-}(x)\overline{\zeta_N^{-}(x)}
\\
= 2 &+& 2\cos(4\pi N x)V_0(x,\alpha)
+\frac{2}{N}\sin(4\pi N x)W_1(x,\alpha) + O(N^{-2}).
\end{eqnarray*}
The first two terms here are summed up as in formula (\ref{cos4Nx}), the last two ones -- as in (\ref{Neumann}) taking into account the relation (\ref{VW}).

Now we consider the last limit in formula (\ref{odd-even}). It can be rewritten in two ways depending on $\alpha$:
\begin{multline*}
({\cal C},1)\,\text{-}\lim_{N\to \infty} y_{2N+1}(x)\overline{z_{2N+1}(x)}\\
=1+\lim_{k \to \infty} \frac{1}{k}\sum \limits_{N=1}^k 
\cos(4\pi N x)V_0(x,\pm\alpha)+\lim_{k \to \infty} \frac{1}{k}\sum \limits_{N=1}^k \sin(4\pi N x)V_1(x,\pm\alpha)\\
+\lim_{k \to \infty} \frac{1}{k}\sum \limits_{N=1}^k \frac{1}{N}\cos(4\pi N x)W_0(x,\pm\alpha)+\lim_{k \to \infty} \frac{1}{k}\sum \limits_{N=1}^k \frac{1}{N}\sin(4\pi N x)W_1(x,\pm\alpha).
\end{multline*}
The constant disappears after integration as before. The second and the third terms are uniformly bounded and converge pointwise. Moreover, the limit equals zero everywhere except points $0$ and $1$ (here we again used the relations (\ref{VW})). The last two terms converge to zero uniformly. By the Lesbegue Theorem, we can pass to the limit under the integral sign. By the assumptions imposed on ${\mathfrak q}$ the endpoints do not contribute to the integral, and formula (\ref{C-lim}) is proved.
\end{proof}

\section{Proof of the main result}

\begin{proof}[Proof of Theorem \ref{th3}]

We start from formula (12) from the paper \cite{SZN}. For $n=2$ and $x=y$ it reads
\begin{equation}\label{227}
G_0(x,x,\lambda)=\tilde{G_0}(x,x,z) = \frac{\Delta_{1,1}(z) + e^{-2izx}\Delta_{1,2}(z) - e^{2izx}\Delta_{2,1}(z) - \Delta_{2,2}(z)}{2iz\Delta(z)}
\end{equation}
 (recall that $z=\lambda^{\frac 12}$). Here $\Delta(z)$ and $\Delta_{\alpha,\beta}(z)$ stand for determinants of order $n$ matrices defined in \cite[Sec. 2.1]{SZN}. In our case they have the following asymptotics as $z\to\infty$:
\begin{multline*}
\Delta(z) = \hat{\Delta}(z)e^{iz(a-b)}(iz)^{d_0+d_1}\cdot(1+O(z^{-1})),\\
\hat{\Delta}(z)=
\begin{vmatrix}
	a_{0}+b_{0}e^{iz(b-a)} & (-1)^{d_0}(a_{0}e^{iz(b-a)}+b_{0}) \\
	a_{1}+b_{1}e^{iz(b-a)} & (-1)^{d_1}(a_{1}e^{iz(b-a)}+b_{1}) \\
\end{vmatrix} ; 
\end{multline*}
\begin{equation*}
\Delta_{1,1}(z) = \hat{\Delta}_{1,1}(z)(iz)^{d_0+d_1}\cdot(1+O(z^{-1})),\qquad
\hat{\Delta}_{1,1}(z)=
\begin{vmatrix}
b_{0} & (-1)^{d_0}(a_{0}e^{iz(b-a)}+b_{0}) \\
	b_{1} & (-1)^{d_1}(a_{1}e^{iz(b-a)}+b_{1}) \\
\end{vmatrix} ;
\end{equation*}
\begin{equation*} \Delta_{1,2}(z) = 
\hat{\Delta}_{1,2} e^{iz(a+b)}(iz)^{d_0+d_1} \cdot(1+O(z^{-1})),\qquad
\hat{\Delta}_{1,2}=
\begin{vmatrix}
	a_{0} & b_{0} \\
	a_{1} & b_{1} \\
\end{vmatrix} ;
\end{equation*}
\begin{equation*}
\Delta_{2,1}(z) = 
\hat{\Delta}_{2,1}e^{-iz(a+b)}(iz)^{d_0+d_1}\cdot(1+O(z^{-1})),\qquad
\hat{\Delta}_{2,1}=(-1)^{d_0+d_1+1}\cdot \hat{\Delta}_{1,2};
\end{equation*}
\begin{multline*} 
\Delta_{2,2}(z) = 
\hat{\Delta}_{2,2}(z)e^{iz(a-b)}(iz)^{d_0+d_1}\cdot(1+O(z^{-1})),\\
\hat{\Delta}_{2,2}(z)=
\begin{vmatrix}
	 a_{0}+b_{0}e^{iz(b-a)} & (-1)^{d_0}b_{0} \\
	 a_{0}+b_{0}e^{iz(b-a)} & (-1)^{d_1}b_{1} \\
\end{vmatrix} .
\end{multline*}
By assumptions $|z|=R$ is separated from $|\lambda_N^0|^{\frac{1}{2}}$. Due to regularity of the boundary conditions (\ref{Boundary_Conditions}) the determinant $\hat{\Delta}(z)$ is separated from zero. Therefore,
\begin{equation*}
\tilde{G_0}(x,x,z)
= \frac{
\hat{\Delta}_{1,1}(z)e^{2iz(b-a)} + 
\hat{\Delta}_{1,2}e^{2iz(b-x)} -\hat{\Delta}_{2,1}e^{2iz(x-a)} -
\hat{\Delta}_{2,2}(z)
}{2iz\hat{\Delta}(z)}\cdot (1+O(z^{-1})).
\end{equation*}
 Since $z=Rw\in R(\Gamma_1 \cup \Gamma_2)$ belongs to the upper half-plane, all exponents are bounded uniformly and tend to zero as $R\to\infty$ for all $x\in (a,b)$ and $w\in \Gamma_1 \cup \Gamma_2$. By the Lebesgue Theorem we obtain
\begin{multline*}
	\lim_{R\to \infty}\int\limits_{|\lambda|=R^2} G_0(x,x,\lambda)^2\,d\lambda = 
    \lim_{R\to \infty}\int\limits_{z=R(\Gamma_1 \cup \Gamma_2)} \!\tilde{G_0}(x,x,z)^2\,2z\,dz \\
= -\lim_{R\to \infty}\int\limits_{z=R(\Gamma_1 \cup \Gamma_2)} \!\Big(\frac{\hat{\Delta}_{2,2}(z)} {\hat{\Delta}(z)}\Big)^2\, \frac{dz}{2z} = -\lim_{R\to \infty}\int\limits_{z=R(\Gamma_1 \cup \Gamma_2)}\frac{dz}{2z}= -\frac{\pi i}{2}.
\end{multline*}
\end{proof}

\begin{proof}[Proof of Theorem \ref{basic-1}]
We consider formula (\ref{Splitting}) and pass to the limit as $R\to\infty$ as it is explained in Theorem \ref{th2}. 

Theorem \ref{th2} shows that the first term in (\ref{Splitting}) converges to the sum of the first two terms in (\ref{SL-SSh-general}). Further, the 	integrand in the second term of (\ref{Splitting}) has a summable majorant in view of the estimate from Proposition \ref{Estimates_G_0}. Theorem \ref{th3} and the Lebesgue Theorem provide the last term in (\ref{SL-SSh-general}).

It is well known (see e.g. \cite{Shk}, or the proof of the Theorem \ref{th2}) that if the boundary conditions (\ref{Boundary_Conditions}) are strongly regular, then the values $|\lambda_N^0|^{\frac{1}{2}}$ are asymptotically separated. Thus in this case passage to the limit in Theorem \ref{th2} corresponds to summation of the series $\mathcal{S}({\mathfrak q})$ by mean-value method, and the statement of Theorem in this case follows.

If the boundary conditions (\ref{Boundary_Conditions}) are regular but not strongly regular then the values $|\lambda_N^0|^{\frac{1}{2}}$ are either pairwise asymptotically close or pairwise coincide. Therefore passage to the limit in Theorem \ref{th2} corresponds to summation of the series $\mathcal{S}({\mathfrak q})$ in the following way: first we add pairwise the asymptotically the terms corresponding to close or coinciding eigenvalues, then the obtained series is summed up by mean-value method.

It remains to notice that $\lambda_N({\mathfrak q}) - \lambda_N \to0$ as $N\to\infty$. Therefore Lemma \ref{pairwise} provides the statement of Theorem in this case.
\end{proof}

\section*{Appendix}
\subsection*{1. The case $d_0=d_1=1$}

The square roots of the eigenvalues of the operator $\mathbb{L}_0$ have the following asymptotics:
$$
\rho_{N+1}:=(\lambda_{N+1}^0)^{\frac 12} = \pi N + \frac{f_1-c_0}{\pi N} + (-1)^N\,\frac{c_1-f_0}{\pi N} + O(N^{-2}).
$$
The eigenfunctions of the operators $\mathbb{L}_0$ and $\mathbb{L}_0^*$ for $\rho_N\neq0$ have the form
\begin{eqnarray*}
y_N(x) &=& C_1\,\Big(\cos(\rho_N x) - c_0\,\frac  {\sin(\rho_N x)}{\rho_N} + f_0\,\frac {\sin(\rho_N (1-x))}{\rho_N}\Big); \\ 
\overline{z_N(x)} &=& C_2\,\Big(\cos(\rho_N x) - c_0 \,\frac {\sin(\rho_N x)}{\rho_N} - c_1 \,\frac  {\sin(\rho_N (1-x))}{\rho_N}\Big).
\end{eqnarray*}
The asymptotics of eigenfunctions is given by
\begin{eqnarray*}
y_{N+1}(x) &=& C_1\,\Big( \cos(\pi N x) - \sin(\pi N x)\,\frac{c_0 (1-x) + f_1 x}{\pi N}\\
&&- (-1)^N\sin(\pi N x)\,\frac{f_0 (1-x)+c_1 x}{\pi N}\Big) + O(N^{-2});\\
\overline{z_{N+1}(x)} &=& C_2\,\Big( \cos(\pi N x) - \sin(\pi N x)\, \frac{c_0 (1-x)+f_1 x}{\pi N}\\ 
&&+ (-1)^N \sin(\pi N x)\,\frac{c_1(1- x) + f_0 x}{\pi N}\Big)  + O(N^{-2}).
\end{eqnarray*}
The asymptotics of scalar products is
$$
\langle y_{N+1}, z_{N+1}\rangle = \frac{C_1C_2}{2} + O(N^{-2}).
$$

\subsection*{2. The case $d_0=0$, $d_1=1$. Jordan blocks}

Recall that we consider the case $\mathfrak C=-\mathfrak A$. In this case the operators $\mathbb{L}_0$ and $\mathbb{L}_0^*$ have the eigenvalue $\lambda_1^0=0$ corresponding to eigenfunctions $y_1(x)=x-\frac{a_0}{a_0+b_0}$ and $z_1(x)\equiv const$. The constant is chosen to meet the normalization conditions. All other eigenvalues are $\lambda_{2N}^0=\lambda_{2N+1}^0=(2\pi N)^2$, $N\in\mathbb{N}$.

The corresponding eigenfunctions and adjoined functions satisfying the first pair of conditions (\ref{prisoed}) are as follows:
\begin{eqnarray*}
y_{2N}(x) &=& C_1\sin(2\pi N x),\\
\widehat y_{2N+1}(x) &=& C_1\Big(\frac{x \cos(2 \pi N x)}{4 \pi N} + \frac{\sin(2 \pi N x)}{16 \pi^2 N^2} - \frac{b_0 \cos(2 \pi N x)}{4 \pi N (a_0 + b_0)}\Big); 
\end{eqnarray*}
\begin{eqnarray*}
\overline{z_{2N}(x)} &=& C_2\cos(2\pi N x),\\
\overline{\widehat z_{2N+1}(x)} &=& C_2\Big(-\frac{x \sin(2 \pi N x)}{4 \pi N} - \frac{\cos(2 \pi N x)}{16 \pi^2 N^2} + \frac{a_0 \sin(2 \pi N x)}{4 \pi N (a_0 + b_0)}\Big). 
\end{eqnarray*}
Scalar products: 
\begin{equation*}
\langle y_{2N}, \widehat z_{2N+1}\rangle = \langle \widehat y_{2N+1}, z_{2N}\rangle = \frac{C_1C_2(a_0-b_0)}{16 \pi N(a_0 + b_0)}
\end{equation*}
(notice that $a_0 \neq b_0$ since $\mathfrak A \neq 0$).

\subsection*{3. The case $d_0=0$, $d_1 = 1$. Asymptotically close eigenvalues}

Recall that we again consider the case $\mathfrak C=-\mathfrak A$. In this case all the eigenvalues of $\mathbb{L}_0$ except for $\lambda_1^0$ pair up,  $\lambda_{2N}$ and $\lambda_{2N+1}$, $N\in\mathbb{N}$, which come close as $N\to\infty$.
Denote the square roots of the eigenvalues of these pairs by $\rho_N^{\pm}$. Furthermore, one of them equals $2\pi N$ (without loss of generality let it be $\rho_N^{+}$). Notice that the scalar products of the corresponding eigenfunctions (we denote them by $\eta_N^{\pm}$ and $\zeta_N^{\pm}$) tend to zero as $N\to\infty$. Therefore we write down the asymptotic formulas of eigenvalues $\rho_N^{-}$ and eigenfunctions $\eta_N^{-}$, $\zeta_N^{-}$ up to $O(N^{-4})$ in the variant 3.1 and up to $O(N^{-6})$ in the variant 3.2. 

\subsubsection*{3.1. Variant $a_0 + b_0 = 0$, $a_1 + b_1 \neq 0$
}

In this variant we obtain
\begin{equation*}
\rho_N^{-} = 2 \pi N + \frac{\mathfrak B}{\pi N \mathfrak A} - \frac{6\mathfrak A \mathfrak B^2+\mathfrak B^3}{12 \mathfrak A^3 N^3 \pi^3}+
O(N^{-4}). 
\end{equation*}
The eigenfunctions of the operators $\mathbb{L}_0$ and $\mathbb{L}_0^*$ have the form
\begin{eqnarray*}
\eta_N^{+}(x) &=& C_1\Big((a_1+ b_1) \cos(2 \pi Nx) - (c_1+f_1)\,\frac{\sin(2 \pi N x)}{2 \pi N}\Big), \\
\overline{\zeta_N^{+}(x)} &=& C_2\sin(2 \pi N x);
\end{eqnarray*}
\begin{eqnarray*}
\eta_N^{-}(x) &=& C_1\Big(a_1 \cos(\rho_N^{-} x) + b_1 \cos( \rho_N^{-}(1-x))- c_1\,\frac{\sin(\rho_N^{-} x)} {\rho_N^{-}} + f_1\, \frac{\sin(\rho_N^{-}(1-x))} {\rho_N^{-}}\Big), \\
\overline{\zeta_N^{-}(x)} &=& C_2\Big(b_1 \sin(\rho_N^{-} x) -a_1 \sin(\rho_N^{-}(1-x))\Big).
\end{eqnarray*}
The asymptotics of scalar products:
\begin{eqnarray*}
\langle \eta_N^+,\zeta_N^+ \rangle &=& -\frac{C_1C_2(c_1+f_1)} {4 \pi N};\\
\langle \eta_N^-,\zeta_N^- \rangle &=& C_1C_2 \Big(\frac{(c_1 + f_1) (a_1 + b_1)}{4 \pi N} - \frac{(c_1+f_1)^2(a_1f_1+c_1b_1)}{8(a_1-b_1)^2\pi^3N^3}\Big) + O(N^{-4})
\end{eqnarray*}
(notice that $a_0 \neq b_0$ since $\mathfrak A \neq 0$).

\subsubsection*{3.2. Variant $a_0 + b_0 = 0$, $a_1 + b_1 = 0$}

In this variant we obtain
\begin{equation*}
\rho_N^{-} = 2 \pi N - \frac{c_1}{2\pi N} + \frac{c_1^3-12c_1^2}{96\pi^3 N^3}+\frac{c_1^4-6c_1^3}{96\pi^5 N^5} + O(N^{-6}).
\end{equation*}
The eigenfunctions of the operators $\mathbb{L}_0$ and $\mathbb{L}_0^*$ have the form
\begin{equation*}
\eta_N^+(x) = C_1 \sin(2\pi N x),\qquad \overline{\zeta_{N}^+(x)} = C_2 \sin(2\pi N x); 
\end{equation*}
\begin{eqnarray*}
\eta_N^-(x) &=& C_1\Big(\sin(\rho^-x) +\sin (\rho^-(1- x))\Big), \\
\overline{\zeta_{N}^-(x)} &=& C_2\Big(\sin(\rho^-x) + \sin (\rho^(1- x))\Big).
\end{eqnarray*}
The asymptotics of scalar products:
\begin{eqnarray*}
\langle \eta_N^+,\zeta_N^+ \rangle &=& \frac{C_1C_2}{2};\\
\langle \eta_N^-,\zeta_N^- \rangle &=& C_1C_2 \Big(\frac{c_1^2}{8\pi^2 N^2}-\frac{c_1^4-4c_1^3}{128\pi^4N^4}\Big) + O(N^{-6}).
\end{eqnarray*}

\subsection*{4. The case $d_0=0$, $d_1=1$. Separated eigenvalues}

In this case the square roots of the eigenvalues $\lambda_{2N}$, $\lambda_{2N+1}$, $N\in\mathbb{N}$ of the operator $\mathbb{L}_0$ form two sequences asymptotically close to two different arithmetic progressions with arithmetical ratio $2\pi$. Denote these roots by $\rho_N^{\pm}$. Then we have, as $N\to\infty$,
\begin{equation*}
\rho_N^\pm = 2 \pi N \pm \alpha + \frac{\mathfrak B}{2\pi N\mathfrak A} + O(N^{-2}),
\end{equation*}
where
\begin{equation*}
\alpha = i\log\left(-\frac{\mathfrak C}{\mathfrak A} - \sqrt{\left(\frac{\mathfrak C}{\mathfrak A}\right)^2-1}\right),
\end{equation*}
and the branch of the logarithm is chosen so that $|\Re(\alpha)|<\pi$
(the choice of another branch yields only to renumbering of eigenvalues). Notice that the condition $\mathfrak C\ne \pm \mathfrak A$ implies $\sin(\alpha) \neq 0$.

The eigenfunctions of the operators $\mathbb{L}_0$ and $\mathbb{L}_0^*$ have the form
\begin{eqnarray*}
\eta_N^\pm(x) &=& a_0 \sin(\rho_N^\pm x) - b_0 \sin(\rho_N^\pm(1-x)), \\
\overline{\zeta_N^\pm(x)} &=& b_0 \cos(\rho_N^\pm x) + a_0 \cos(\rho_N^\pm (1-x)) + \mathfrak B\,\frac{ \sin(\rho_N^\pm x)}{a_1 \rho_N^\pm}.
\end{eqnarray*}
The asymptotics of scalar products:
\begin{equation*}
\langle \eta_N^\pm,\zeta_N^\pm \rangle = \pm\sin(\alpha)\frac{a_0^2 - b_0^2}{2} + \mathfrak B\frac{\mathfrak A a_0 + (\mathfrak A b_0 + a_1(a_0^2-b_0^2)) \cos(\alpha)}{4\pi \mathfrak A a_1 N} + O(N^{-2})
\end{equation*}
(notice that $a_0 \neq \pm b_0$ since $\mathfrak C \neq \pm\mathfrak A$).

The normalized products have the following asymptotics:
\begin{eqnarray*}
\eta_N^\pm(x)\overline{\zeta_N^\pm(x)} = 1 &+& \cos(4\pi N x)V_0(x,\pm\alpha)+\sin(4\pi N x)V_1(x,\pm\alpha)\\
&+& \frac{1}{N}\cos(4\pi N x)W_0(x,\pm\alpha)+\frac{1}{N}\sin(4\pi N x)W_1(x,\pm\alpha) + O(N^{-2}),
\end{eqnarray*}
where
\begin{equation*}
V_0(x,\alpha) = \frac{\sin(\alpha (2x - 1))}
{(a_0^2 - b_0^2)\sin(\alpha)}\,
(a_0^2 + b_0^2+2 a_0 b_0 \cos(\alpha));
\end{equation*}
\begin{equation*}
V_1(x,\alpha) = \frac{\cos(\alpha (2x - 1))}
{(a_0^2 - b_0^2)\sin(\alpha)}\,
(a_0^2 + b_0^2+2 a_0 b_0 \cos(\alpha));
\end{equation*}
\begin{equation*}
W_0(x,\alpha) = \frac{\mathfrak B(2 {\cal R}_1 \sin(\alpha)\cos(2 \alpha x) - {\cal R}_2 \sin(2 \alpha x))}
{4 \mathfrak A a_1 (a_0^2 - b_0^2)^2 \pi\sin^2(\alpha)};
\end{equation*}
\begin{equation*}
W_1(x,\alpha) = -\frac{\mathfrak B(2 {\cal R}_1 \sin(\alpha)\sin(2 \alpha x) + {\cal R}_2 \cos(2 \alpha x))}
{4 \mathfrak A a_1 (a_0^2 - b_0^2)^2 \pi\sin^2(\alpha)};
\end{equation*}
\begin{eqnarray*}
{\cal R}_1 &=&  a_0 b_0 (3 \mathfrak A b_0 + a_1 (a_0^2 - b_0^2) (1 + 2 x)) + 2 (a_0^2 + b_0^2) (\mathfrak A b_0 + a_1 (a_0^2 - b_0^2) x) \cos(\alpha) \\  
&+& a_0 b_0 (\mathfrak A b_0 + a_1 (a_0^2 - b_0^2) (2x-1)) \cos(2 \alpha); 
\end{eqnarray*}
\begin{multline*}
{\cal R}_2 = 4 \mathfrak A a_0^2 b_0 + 2 a_1 (a_0^4 - b_0^4) (1-x) \\
+ a_0 (\mathfrak A(2 a_0^2 + 5 b_0^2)  
+ a_1 b_0 (a_0^2 - b_0^2) (5 - 2 x)) \cos(\alpha)\\ 
+ 2 (a_0^2 + b_0^2) (\mathfrak A b_0 + a_1 (a_0^2 - b_0^2) x) \cos(2 \alpha) \\ 
+ a_0 b_0(\mathfrak A b_0  + a_1 (a_0^2 - b_0^2) (2x-1)) \cos(3 \alpha).
\end{multline*}

\begin{lemma} The functions $V_0(x,\alpha)$, $V_1(x,\alpha)$, $W_0(x,\alpha)$ and $W_1(x,\alpha)$ are continuous in both variables if $\sin(\alpha) \neq 0$ and satisfy the following identities:
\begin{eqnarray}
V_0(x,\alpha) \equiv V_0(x,-\alpha),& \quad &
V_1(x,\alpha) \equiv -V_1(x,-\alpha);
\nonumber\\
W_1(x,\alpha) \equiv W_1(x,-\alpha),& \quad &
W_0(x,\alpha) \equiv -W_0(x,-\alpha);
\nonumber\\
V_0({\textstyle \frac{1}{2}},\alpha) \equiv 0, &&
W_1({\textstyle \frac{1}{2}},\alpha) \equiv 0.
\label{VW}
\end{eqnarray}
\end{lemma}

\begin{proof}
Since ${\cal R}_1$ and ${\cal R}_2$ are even functions of $\alpha$, all statements of the Lemma except for the last one are obvious. Taking into account the relation $\mathfrak A\cos(\alpha)+\mathfrak C=0$ we obtain that the numerator in $W_1$ can be rewritten as follows:
\begin{eqnarray*}
2 {\cal R}_1 \sin(\alpha)\sin(2 \alpha x) &+& {\cal R}_2 \cos(2\alpha x) =  2 \sin(\alpha (x-{\textstyle\frac{1}{2}}))\\
 &\times&\Big( \big(a_0 a_1 b_0 (a_0^2-b_0^2)(2 x-1)+a_0b_0^2\mathfrak A\big) \sin(\alpha ({\textstyle\frac{5}{2}} - x))  \\
&+& \big(2a_1 (a_0^4 - b_0^4)x+2b_0(a_0^2 + b_0^2)\mathfrak A\big) \sin(\alpha ({\textstyle\frac{3}{2}} - x)) \\
&-& \big(5 a_0^3 a_1 b_0 + a_0 a_1 b_0^3 + a_0^4 b_1 + 5 a_0^2 b_0^2 b_1\big) \sin(\alpha (x- {\textstyle\frac{1}{2}})) \\
&-& \big(2 a_1(a_0^4- b_0^4)(1-x) + 4 a_0^2 b_0\mathfrak A\big) \sin(\alpha ({\textstyle\frac{1}{2}} + x)) \\
&+& \big(a_0 a_1 b_0(a_0^2- b_0^2)(2x-1)-a_0^3\mathfrak A\big) \sin(\alpha ({\textstyle \frac{3}{2}} + x))\Big),
\end{eqnarray*}
and the last equality in (\ref{VW}) is proved.
\end{proof}

\section*{Acknowledgments}

The main results of the paper, Theorems 2.4 and 2.5, are obtained under support of the Russian Science Foundation grant N14-21-00035. 
Theorem 2.3 was obtained under support of RFBR grant 16-01-00258a.


\begin{thebibliography}{XX}

\bibitem{GL} 
I. M. Gelfand, B. M. Levitan, {\em On a simple identity for the eigenvalues of a second-order differential operator}, DAN SSSR, {\bf 88} 
(1953), 593--596 (Russian). 

\bibitem{Konechnaya}
N. N. Konechnaya, T. A. Safonova, R. N. Tagirova, {\em Asymptotics of eigenvalues and regularised first-order trace of the Sturm-Liouville operator with
$\delta$-potential}, Vestnik SAFU, 2016, N1, 104--113 (Russian).

\bibitem{Neimark}
M. A. Naimark, {\em Linear differential operators}, ed.2, Moscow, Nauka, 1969 (Russian). English transl. of the first ed.: 
{\em Linear Differential Operators, V.1: Elementary theory of linear differential operators}, Harrap, 1967.

\bibitem{SZN}
A. I. Nazarov, D. M. Stolyarov, P. B. Zatitskiy, {\em The Tamarkin equiconvergence theorem and a first-order trace formula for regular differential operators revisited}, 
J. Spectral Theory, {\bf 4} (2014), N2, 365--389.  

\bibitem{Prudnikov}
A. P. Prudnikov, Yu. A. Brychkov, O. I. Marichev, {\em Integrals and Series. Elementary Functions}, Moscow, Nauka, 1981 (Russian). English transl.: 
{\em Integrals and Series, V.I: Elementary Functions}, Gordon and Breach Science Publishers, 1998.

\bibitem{SPSurvey}
V. A. Sadovnichii, V. E. Podolskii, {\em Traces of operators}, Uspekhi Mat. Nauk, {\bf 61} (2006), N5, 89--156 (Russian). English transl.: 
Russian Math. Surveys, {\bf 61} (2006), N5, 885--953.

\bibitem{S2000}
A. M. Savchuk, {\em First-order regularised trace of the Sturm-Liouville operator with $\delta$-potential}, Uspekhi Mat. Nauk, {\bf 55} (2000), N6, 155--156 (Russian).
English transl.: Russian Math. Surveys, {\bf 55} (2000), N6, 1168--1169.

\bibitem{SSh}
A. M. Savchuk, A. A. Shkalikov, {\em Trace Formula for Sturm--Liouville Operators with Singular Potentials}, 
Matem. Zametki, {\bf 69} (2001), N3, 427--442 (Russian). English transl.: Math. Notes, {\bf 69} (2001), N3, 387--400.

\bibitem{Shv}
R. F. Shevchenko, {\em On the trace of a differential operator}, DAN SSSR, {\bf 164} (1965), N1, 62--65 (Russian). English transl.: Soviet Math. Dokl., 
{\bf 6} (1965), 1183--1186.

\bibitem{Shk}
A. A. Shkalikov, {\em Boundary-value problems for ordinary differential equations with a parameter in the boundary conditions}, 
Trudy Seminara I.G. Petrovskogo, {\bf 9} (1983), 190--229 (Russian). English transl.: J. Soviet Math. {\bf 33} (1986), 1311--1342.

\bibitem{T}
J. D. Tamarkin, {\em On some general problems of the theory of ordinary linear differential
operators and on expansion of arbitrary function into serii}, Petrograd.
1917, 308 p. (Russian)


\end{thebibliography}
\end{document}